\documentclass[12pt,a4paper,reqno]{amsart}

\usepackage{graphicx} 
\usepackage{amsfonts}
\usepackage{amsmath}
\usepackage{amsthm}
\usepackage{amssymb}
\usepackage{mathtools}
\usepackage{geometry}

\usepackage{xcolor}

\newcommand{\RR}{\mathbb{R}}

\newtheorem{theorem}{Theorem}[section]

\theoremstyle{remark}
\newtheorem*{remark}{Remark}

\begin{document}

\title{Moments of Characteristic Polynomials at Random Matrix Gram Points}
\author{Christopher Hughes}
\email{christopher.hughes@york.ac.uk}
\address{Department of Mathematics, University of York, York, YO10 5GH, United Kingdom}

\author{Somin Kim}
\email{ttb535@york.ac.uk}
\address{Department of Mathematics, University of York, York, YO10 5GH, United Kingdom}

\date{}

\begin{abstract}
    We evaluate the characteristic polynomial of a random unitary matrix at the points on the unit circle where the polynomial takes real values. In analogy to the Riemann zeta function we call these the ``random matrix Gram points''. On the random matrix side we answer in the affirmative a question of Richard Hall, who expected that zeta on the Gram points should be larger than typical. We prove that for any fixed positive $k$, the $2k$th moment of the characteristic polynomial averaged over  the random matrix Gram points is larger than the  $2k$th moment of the characteristic polynomial averaged over the whole unit circle. For large $N$ the leading-order asymptotic term is the same size, but all lower-order terms are larger. In an appendix we show how to adapt the heuristic recipe of Conrey--Farmer--Keating--Rubinstein--Snaith to conjecture moments of zeta at Gram points, and show by example that the second and fourth moments calculated by following the recipe agree with previously proven results.
\end{abstract}

\maketitle

\section{Introduction}

One of the key unsolved problems concerning the Riemann zeta function is bounding its growth. Most famously the Lindel\"of Hypothesis says that $\zeta(1/2+it)$ grows no faster than $t^\epsilon$ for large $t$. A common way to attack this problem is by studying moments, as high moments of a function tell us information about large values of that function.

Building on over a century of research, Conrey--Farmer--Keating--Rubinstein--Snaith \cite{CFKRS05} created a heuristic argument, known as ``the recipe'', to argue that for $k$ a positive integer, there exists an explicit degree $k^2$ polynomial, $\widetilde K_k$, such that
\begin{equation}\label{eq:CFKRS}
\int_{-\infty}^\infty \omega\left(t\right) \left| \zeta(\tfrac12+it) \right|^{2k} dt \sim \int_{-\infty}^\infty \omega\left(t\right) \widetilde K_k\left(\log\frac{t}{2\pi}\right) \, dt
\end{equation}
where $\omega$ is a weight function depending on some asymptotic parameter, typically taken to be the indicator function on $[0,T]$ for large $T$, or sometimes $e^{-\delta t}$ for small $\delta$. The conjecture is known to be true in the case $k=1$ and $k=2$ only. For $k=1$ combining and reinterpreting the results from a number of papers \cite{HL18,Ing26} show
\begin{equation}\label{eq:SecondCntsMmt}
    \widetilde K_1(x) = x + 2\gamma
\end{equation}
and for $k=2$ the papers \cite{Ing26,HB79,Con96} show
\begin{equation}\label{eq:FourthCntsMmt}
\widetilde K_2(x) = \frac{1}{2\pi^2} x^4 + \left(\frac{8 \gamma }{\pi ^2}-\frac{24 \zeta '(2)}{\pi ^4}\right) x^3 + a_2 x^2 + a_1 x + a_0
\end{equation}
where $a_2, a_1, a_0$ are all explicit constants, too long to write here, but implicitly given in the papers above.

Intimately connected to the maximum size of zeta are the zeros of the zeta function. It is known that the number of zeros of zeta in the critical strip with height between $0$ and $T$ equals
\[
N(T) = \frac{1}{\pi} \theta(T) + 1 + S(T)
\]
where 
\begin{align*}
    \theta(T) &=  \Im \left( \log \Gamma \left( \frac{1}{4} + \frac{i T}{2} \right) \right) - \frac{T}{2} \log \pi \\
    &=\frac{T}{2} \log \left( \frac{T}{2\pi e} \right)  - \frac{\pi}{8} + O\left(\frac{1}{T}\right)
\end{align*}
is the ``smooth'' part of the counting function, and $S(T) = O(\log T)$ is the oscillatory part.

The $m$th Gram point, $g_m$ is the value $\geq 10$ such that $\theta(g_m) = m \pi$. Due to the smooth nature of $\theta(t)$, these points are very regularly spaced. Whilst it was Gram \cite{Gram1903} who numerically observed that the first few zeros of zeta interlace with the Gram points (which is another way of saying that $S(T)$ in the counting function stays small for $T$ within computable ranges), it was Hutchinson \cite{Hutchinson25} who coined the term ``Gram's Law'' for this phenomenon. Hutchinson went on to say that ``Gram expressed the belief that this law is not a general one'', and indeed we now know it regularly fails. However, it is highly likely that a positive proportion of the time  between two consecutive zeros of zeta there lies exactly one Gram point.

If a Gram points lies roughly in the middle of two consecutive zeros of zeta, then Richard Hall (in a conversation with the first author) wondered if zeta evaluated at the Gram points should be bigger than a typical value of zeta.

In order to answer this, we will use another equally well-known interpretation of Gram points: namely they are the values of $t$ which make $\zeta(\tfrac12+it)$ real. This can easily be seen to be equivalent to the prior definition of Gram points by noting that the functional equation,
\[
\zeta(s) = \chi(s) \zeta(1-s)
\]
has $\chi(\tfrac12 + i t) = e^{-2i\theta(t)}$ with the same function $\theta$ seen above.

This actually allows us to define a more general phi-Gram point for $0\leq \phi < \pi$, which we denote by $g_m(\phi)$. These are the points where  $\theta(g_m(\phi)) = m \pi + \phi$. By a similar rephrasing to the above, they are the points where
\[
\zeta(\tfrac12 +i g_m(\phi)) \in \RR e^{i\phi}
\]
Note that $g_m(0)$ (that is, when $\phi=0$) are the usual Gram points. 

Kalpokas and Stueding~\cite{KalSte09} showed that the second moment when evaluated at the phi-Gram points satisfies
\begin{multline*}
\sum_{g_m(\phi) \leq T } \left| \zeta(\tfrac12+i g_m(\phi)) \right|^2  
= \frac{T}{2\pi} \left( \left( \log \frac{T}{2\pi e} \right)^2 + (2\gamma+2\cos(2\phi)) \log \frac{T}{2\pi e} + 1 \right) \\ + O\left(T^{1/2+\epsilon}\right)
\end{multline*}
We observe, possibly for the first time, that the terms on the right-hand side can be written as
\[
\int_0^T  K_1^{\phi}\left(\log\frac{t}{2\pi}\right) \frac{1}{2\pi}\log \frac{t}{2\pi} \, dt
\]
where
\begin{equation}\label{eq:Kernelk1}
K_1^{\phi}(x) = x + 2\gamma+2\cos(2\phi)
\end{equation}

Note that $\frac{1}{2\pi} \log \frac{t}{2\pi}$ is the smooth density of the phi-Gram points up the critical line. If we compare the kernels for the continuous and discrete second moments (given in \eqref{eq:SecondCntsMmt} and \eqref{eq:Kernelk1}, respectively), we see that they have the same leading order asymptotic, but the lower order terms are different in size. When $\phi=0$ (that is, the Gram points), the kernel $K_1^0(x)$ is larger than $\widetilde K_1(x)$.

For the fourth moment, Lavrik~\cite{Lavrik95} showed
\[
\sum_{0<g_m \leq T}  \left| \zeta(\tfrac12+i g_m) \right|^4 = \frac{T}{4 \pi^3} (\log T)^5 + O \left(T (\log T)^{9/2}\right)
\]
which again, when dividing by $N(T)$ has the same asymptotic size as the continuous fourth moment.

Hung Bui and Richard Hall \cite{BuiHall} have shown (in an unpublished work still in preparation, but kindly shared with the authors) that there is a degree-5 polynomial yielding all the lower order terms, proving
\[
\sum_{0<g_m \leq T}  \left| \zeta(\tfrac12+i g_m) \right|^4 = T \sum_{j=0}^5 a_j (\log T)^j + O \left(T^{2/3+\epsilon}\right)
\]
with $a_5 = \frac{T}{4 \pi^3}$. The lower coefficients get more complicated to write, but are the $\phi=0$ specialisation of
\begin{equation}\label{eq:FourthMmtGram}
\sum_{0<g_m(\phi) \leq T}  \left| \zeta(\tfrac12+i g_m(\phi)) \right|^4 \sim \int_0^T K_2^{\phi}\left(\log\frac{t}{2\pi}\right) \frac{1}{2\pi}\log \frac{t}{2\pi} \, dt 
\end{equation}
where 
\[
K_2^\phi(x) = \widetilde K_2(x) +2\cos(2\phi) L_2(x) + 2 \cos(4\phi)
\]
where $\widetilde K_2(x)$ is the kernel for the continuous fourth moment, given in \eqref{eq:FourthCntsMmt}, and 
\begin{equation}\label{eq:L2}
L_2(x) = \frac{1}{6} x^3 +  2 \gamma  x^2 + \left(6 \gamma
   ^2-4 \gamma _1 \right)x + \left(4 \gamma ^3 -12 \gamma  \gamma _1+2 \gamma _2 \right)
\end{equation}
We note that all four coefficients in the cubic polynomial $L_2(x)$ are positive. This result will be discussed in Appendix~\ref{sect:k2Appendix}.

\begin{remark}
    Note that the Bui--Hall result is rigorously proven, whereas our prediction in the appendix comes from reinterpreting the heuristic recipe methodology.
\end{remark}

The purpose of this paper is to use random matrix theory to argue that the answer to Hall's question, whether zeta is larger at the Gram points than typical, is affirmative: we conjecture that the $2k$th moment of the Riemann zeta function at the Gram points is always bigger than the $2k$th continuous moment, but only in the sub-leading terms. We anticipate the leading order term will be the same size as the continuous moments. We arrive at our conjecture by analysing an analogous problem in random matrix theory.

\subsection{The random matrix model and our results}

The conjecture \eqref{eq:CFKRS} partly grew out of the work by Keating and Snaith~\cite{KS00a} on modelling the Riemann zeta function using random matrix theory, specifically the characteristic polynomial
\[
Z_U(\theta) = \det\left(I - U e^{-i\theta}\right)
\]
of a Haar-distributed random unitary matrix.

Keating and Snaith proved that the $2k$th moment of this (for $\Re(k)>-1/2)$, taken over all $N\times N$ unitary matrices chosen with Haar measure equals
\begin{align*}
\mathbb{E}\left[ Z(0)|^{2k} \right] &= \prod_{j=1}^N \frac{\Gamma(j) \Gamma(j+2k)} {\Gamma(j+k)^2} \\
&=\frac{G(1+k)^2}{G(1+2k)} N^{k^2}
\end{align*}
where $\Gamma$ is the usual Gamma function, and $G$ is Barnes $G$-function, a meromorphic function which satisfies $G(z+1) = \Gamma(z) G(z)$. 

The aim of this paper is to investigate moments of characteristic polynomials at Gram points in the random matrix theory model. We compare these moments with the classical CUE averages over \(N\times N\) unitary matrices. Although the leading-order main term agrees with that of the 
 \(\mathbb{E}[|Z(0)|^s]\), we show that the moments evaluated at Gram points are in fact larger. We now state the main theorems of this paper.

First, building on the work of Hanga~\cite{Hanga20} we define the random matrix phi-Gram points.
\begin{theorem} \label{thm:GramPoints}
Let $U$ be an $N \times N$ unitary matrix with $\det(U) = e^{i\omega}$. For $0 \leq \phi < \pi$, let
\[
g_m^\phi = \frac{\omega}{N} + \pi N  - \frac{2\phi}{N} + \frac{2\pi m}{N},
\]
for $m=1,\dots,N$. At the points $g_m^\phi$, the characteristic polynomial of $U$ satisfies \[Z_U(g_m^\phi) \in \RR e^{i\phi} .\]
\end{theorem}

Next we will evaluate the characteristic polynomial evaluated at these phi-Gram points, averaged over Haar-distributed $N\times N$ unitary matrices.
\begin{theorem}\label{thm:ExactMmtAbsVal}
Let \(g_m^\phi\) denote the random matrix phi-Gram points and let \(Z\) be the characteristic polynomial. Taking the expectation over $N \times N$ unitary matrices with respect to Haar measure, one finds that for $\Re(s)>-1$
\[
\mathbb{E}\left[
\frac{1}{N} \sum_{m=1}^N \left| Z(g_m^\phi) \right|^{s}\right]
= \mathbb{E}[\left| Z(0) \right|^{s}] \left( 1 + 2\sum_{\ell=1}^\infty \cos(2\ell\phi)C_\ell^{N}(s) \right)
\]
where
\begin{equation}\label{eq:CellN}
C_\ell^N(s) = \prod_{j=1}^{N}
\frac{\Gamma\!\left(\tfrac{s}{2} + j\right)^2 }{\Gamma\!\left(\tfrac{s}{2} + \ell + j\right)\Gamma\!\left(\tfrac{s}{2} - \ell + j\right)}.
\end{equation}
where, in the case $s$ is an even integer, we interpret $C_\ell^N(s) = 0$ for $\ell>s/2$.
\end{theorem}

From that theorem we deduce our two main results.
\begin{theorem}\label{thm:first-main}
For  \(\Re(s)>-1\) and for any $0\leq \phi < \pi$, we have
\[
\mathbb{E}\left[
\frac{1}{N} \sum_{m=1}^N \left| Z(g_m^\phi) \right|^{s}\right]
= \mathbb{E}[\left| Z(0) \right|^{s}] \left( 1 + O\left(\frac{1}{N}\right) \right).
\]
\end{theorem}
This theorem shows that the main term of the moments at Gram points agrees with that of the corresponding CUE moments.
\begin{theorem}\label{thm:second-main}
For $\phi=0$ and for all real and positive \(s\),
\[
\mathbb{E}\left[\frac{1}{N} \sum_{m=1}^N \left| Z(g_m^0) \right|^{s}\right] \geq  \mathbb{E}[\left| Z(0) \right|^{s}] .
\]
\end{theorem}
This theorem shows that the moments at Gram points are larger than the corresponding CUE moments.

\section{Random Matrix Gram Points}

We introduce a notion of Gram points in the random matrix setting, motivated by the classical Gram points associated with the zeta function. Since the Riemann zeta function takes real values at its Gram points after a suitable normalization, we study the corresponding question for characteristic polynomials of random matrices. Using the functional equation satisfied by these characteristic polynomials, we determine the points at which they become real, thereby identifying the random matrix analogue of Gram points.

\begin{proof}[Proof of Theorem~\ref{thm:GramPoints}]
Note that
\begin{align*}
Z(\theta) &=\det(I - U e^{-i \theta}) \\
&= (-1)^N e^{-i N \theta} \det(U) \det(I-U^* e^{i\theta}) \\
&= (-1)^N e^{-i N \theta} \det(U) \overline{Z_U(\theta)} \\
&= \exp\!\left(i\big(\pi N - N\theta + \omega\big)\right) \overline{Z(\theta)}
\end{align*}
since  \((-1)^Ne^{-i N \theta} = e^{i(\pi N - N\theta)}\) and \(\det(U)=e^{i\omega}\).

Therefore, if we define
\[
g_m^\phi = \frac{\omega}{N} + \pi N - \frac{2\phi}{N} + \frac{2\pi m}{N} 
\]
for any integer $m$  (we note $g_m^\phi$ are all distinct modulo $2\pi$ for $m=1,\dots,N$), then
\[
Z(g_m^\phi) = e^{2i\phi} \overline{Z(g_m^\phi)} 
\]
so $Z(g_m^\phi) \in \RR e^{i\phi}$ at such points. We call $g_m^\phi$ the random matrix phi-Gram points.
\end{proof}

\section{Selberg's integral}
Our first step is to employ Selberg’s integral technique in order to derive a series representation for the moments. 

We will show that for $\Re(s)>-1$
\[
\mathbb{E}\left[
\frac{1}{N} \sum_{m=1}^N \left| Z(g_m^\phi) \right|^{s}\right]
= \mathbb{E}[\left| Z(0) \right|^{s}] \left( 1 + 2\cos(2\phi)C_1^{\,N}(s) + 2\cos(4\phi)C_2^{\,N}(s) + \cdots \right)
\]
where
\begin{equation*}
C_\ell^N(s) = \prod_{j=1}^{N}
\frac{\Gamma\!\left(\tfrac{s}{2} + j\right)^2 }{\Gamma\!\left(\tfrac{s}{2} + \ell + j\right)\Gamma\!\left(\tfrac{s}{2} - \ell + j\right)}.
\end{equation*}

\begin{proof}[Proof of Theorem~\ref{thm:ExactMmtAbsVal}]
Note that \(|Z(\theta)|^s\) is a $2\pi$-periodic function of $\theta$, and so has a Fourier series expansion
\[
|Z(\theta)|^s =  \sum_{k=-\infty}^\infty \widehat{f}_k e^{i k \theta}
\]
where 
\begin{equation}\label{eq:FourierCoeff}
\widehat{f}_{k}=\frac{1}{2\pi} \int_0^{2\pi} |\det(I-U e^{-i\theta})|^s e^{-i k \theta}d\theta.
\end{equation}
so averaging this over the phi-Gram points given in Theorem~\ref{thm:GramPoints},
\begin{align*}
\frac{1}{N}\sum_{m=1}^N |Z(g_m^\phi)|^s
&= \sum_{k \in \mathbb{Z}} \widehat{f}_k \, \frac{1}{N}\sum_{m=1}^N
\exp\!\left( i k \left( \frac{2\pi m}{N} + \frac{\omega}{N} + \pi N - \frac{2\phi}{N}\right) \right) \\
&= \sum_{\ell \in \mathbb{Z}} \widehat{f}_{lN} \, e^{i \ell \omega} \, e^{i \ell N \pi} e^{-2i\ell\phi}\\
&= \sum_{\ell \in \mathbb{Z}} (-1)^{\ell N}\widehat{f}_{\ell N} \, \det(U)^{\ell} e^{-2i\ell\phi}.
\end{align*}

This holds because
\[
\sum_{m=1}^N e^{i k (2\pi m/N)} = 
\begin{cases}
N & \text{ if } k \equiv 0 \pmod{N}\\
0 & \text{ otherwise}
\end{cases}
\]
and $\det(U) = e^{i\omega}$.

Therefore,  we want to compute 
\begin{equation}\label{eq:FourierExpansion}
\mathbb{E}\!\left[ \frac{1}{N}\sum_{m=1}^N |Z(g_m^\phi)|^s \right]=\sum_{\ell \in \mathbb{Z}} (-1)^{N \ell} \mathbb{E}\left[\widehat{f}_{\ell N} (\det U)^{\ell}\right]  e^{-2i\ell\phi}.
\end{equation}
Using the Weyl integration formula and \eqref{eq:FourierCoeff} we have
\begin{multline*}
(-1)^{N \ell} \mathbb{E}\left[\widehat{f}_{\ell N} \, (\det U)^{\ell}\right] =
\frac{1}{2\pi} \int^{2\pi}_0 \frac{1}{(2\pi)^N N!} \int^{2\pi}_0 \cdots \int^{2\pi}_0 \prod_{1\leq j<k\leq N}|e^{i \theta_j}-e^{i\theta_k}|^2 \\
\times \prod_{n=1}^N \left( (-1)^\ell |1-e^{-i\theta}e^{i\theta_n}|^s e^{i\ell \theta_n}\right)d\theta_1\cdots d\theta_N e^{-i\ell N \theta}d\theta.
\end{multline*}
When we shift the variables to \(\theta_n \mapsto \theta_n +\theta\) and use the periodicity of the integrand, the above equation becomes
\[
\frac{1}{(2\pi)^N N!} \int^{2\pi}_0 \cdots \int^{2\pi}_0 \prod_{1\leq j<k\leq N}|e^{i \theta_j}-e^{i\theta_k}|^2 \prod_{n=1}^N \left((-1)^\ell |1-e^{i\theta_n}|^s e^{i\ell \theta_n}\right)d\theta_1\cdots d\theta_N 
\]
because this shift changes  \(\prod|1-e^{i\theta}e^{i\theta_n}|^s \mapsto \prod|1-e^{i\theta_n}|\) and since \(\left( \prod e^{i\ell\theta_n} \right)e^{-i\ell N \theta} \mapsto \prod e^{i \ell \theta_n}\), and the remaining $\theta$-integral is trivial.
Now again shifting variables to \(\theta_n\mapsto\theta_n+\pi\),
\[
\frac{1}{(2\pi)^N N!} \int^{2\pi}_0 \cdots \int^{2\pi}_0 \prod_{1\leq j<k\leq N}|e^{i \theta_j}-e^{i\theta_k}|^2 \prod_{n=1}^N \left( |1+e^{i\theta_n}|^s e^{i\ell \theta_n}\right)d\theta_1\cdots d\theta_N 
\]
absorbs the $(-1)^\ell$ factor.

Using formula (1.18) from \cite{forrester2007importanceselbergintegral}  we plug in
\[
a =  \tfrac{s}{2} + \ell, \qquad b = \tfrac{s}{2} - \ell, \qquad \gamma = 1.
\]
Thus, the quantity we want  is
\begin{multline*}
\frac{1}{(2\pi)^N N!} \int^{2\pi}_0 \cdots \int^{2\pi}_0 \prod_{1\leq j<k\leq N}|e^{i \theta_j}-e^{i\theta_k}|^2 \prod_{n=1}^N \left( |1+e^{i\theta_n}|^s e^{i\ell \theta_n}\right)d\theta_1\cdots d\theta_N  \\
= \frac{1}{N!}\, M_N(a,b,1)
= \frac{1}{N!} \prod_{j=0}^{N-1}
\frac{\Gamma(2+j)\Gamma(1+s+j)}{\Gamma\!\left(\tfrac{s}{2} +1+ \ell +j\right)\Gamma\!\left(\tfrac{s}{2} - \ell +1+ j\right)}
\end{multline*}
Plugging this into \eqref{eq:FourierExpansion} we have shown
\begin{align*}
\mathbb{E}\!\left[ \frac{1}{N}\sum_{m=1}^N |Z(g_m^\phi)|^s \right] &= \sum_{\ell=-\infty}^{\infty} e^{-2i\ell\phi} \frac{1}{N!} \prod_{j=0}^{N-1}
\frac{\Gamma(2+j)\Gamma(1+s+j)}{\Gamma\!\left(\tfrac{s}{2} +1+ \ell +j\right)\Gamma\!\left(\tfrac{s}{2} - \ell +1+ j\right)} \\
&=\sum_{\ell=-\infty}^{\infty} e^{-2i\ell\phi} \frac{1}{N!} \prod_{j=1}^{N}
\frac{\Gamma(1+j)\Gamma(s+j)}{\Gamma\!\left(\tfrac{s}{2} + \ell +j\right)\Gamma\!\left(\tfrac{s}{2} - \ell + j\right)}\\
&= \sum_{\ell=-\infty}^{\infty} e^{-2i\ell\phi} \mathbb{E}[|Z(0)|^s] \prod_{j=1}^{N}
\frac{\Gamma\!\left(\tfrac{s}{2} + j\right)^2 }{\Gamma\!\left(\tfrac{s}{2} + \ell + j\right)\Gamma\!\left(\tfrac{s}{2} - \ell + j\right)}\\
&=\mathbb{E}[|Z(0)|^s](1+2\cos(2\phi)C_1^N(s)+2\cos(4\phi)C_2^N(s) + \cdots)
\end{align*}
for each \(C_\ell^N(s)\) is sum of \(\ell\)th entry of the sum. That is, for $\ell \geq 1$,
\[
C_\ell^{N}(s) = \prod_{j=1}^{N}
\frac{\Gamma\!\left(\tfrac{s}{2} + j\right)^2 }{\Gamma\!\left(\tfrac{s}{2} + \ell + j\right) \Gamma\!\left(\tfrac{s}{2} - \ell + j\right)}
\]
as required.
\end{proof}

\begin{remark}
If $s$ is an even integer, $s=2k$, say, then the result above simplifies to a finite sum, with each term in the summand being evidently positive when $\phi=0$.
\end{remark}

\begin{theorem}\label{thm:EvenInteger}
    When \(s=2k\) is an even integer, we have
  \[
\mathbb{E}\left[\frac{1}{N}\sum_{m=1}^N |Z(g_m^\phi)|^{2k} \right] = \mathbb{E}[|Z(0)|^{2k}]\left(1+2 \sum_{\ell=1}^k \cos(2\ell\phi) C_\ell^N(2k) \right)
\]
where $g_1,\dots,g_N$ are the random matrix Gram points defined in Theorem~\ref{thm:GramPoints}. The terms in the sum, $C_\ell^N(2k)$, are all positive for $1\leq \ell \leq k$.
\end{theorem}
\begin{proof}
For \(s=2k\), an even integer, the terms in the product for $C_\ell^N(2k)$ are
\[
\prod_{j=1}^N\frac{\Gamma(k+j)^2}
{\Gamma(k+j+\ell)\Gamma(k+j-\ell)}.
\]
and since the Gamma function has poles at non-positive integers, when \(|\ell| \geq k+1\) the $j=1$ term means the product vanishes. Hence the infinite sum truncates to finite sum, and we have
\[
\mathbb{E}\left[\frac{1}{N}\sum_{m=1}^N |Z(g_m^\phi)|^{2k} \right] = \mathbb{E}[|Z(0)|^{2k}]\left(1+2 \sum_{\ell=1}^k \cos(2\ell\phi) C_\ell^N(2k) \right).
\]
It is clear that every term in the product for $C_\ell^N(2k)$ is positive for $|\ell| \leq k$, so the entire sum is positive when $\phi=0$.
\end{proof}

\section{Convergence of the sum}
Returning to the general case of $s$ complex with $\Re(s)>-1$, in this section we prove Theorem~\ref{thm:first-main}, and show
\[
\sum^\infty_{\ell=1} C_\ell^N(s) = O\left(\frac{1}{N}\right).
\]

Furthermore, if we specialise to $s$ real and positive, we show that the sum is positive, which proves Theorem~\ref{thm:second-main}.

\begin{proof}[Proof of Theorem~\ref{thm:first-main}]
Recall the expression for $C_\ell^N(s)$, given in \eqref{eq:CellN}. Our aim is to show
\[
\sum_{\ell=1}^{\infty} \prod_{j=1}^{N}
\frac{\Gamma\!\left(\tfrac{s}{2} + j\right)^2 }{\Gamma\!\left(\tfrac{s}{2} + \ell + j\right)\Gamma\!\left(\tfrac{s}{2} - \ell + j\right)} = O\left( \frac{1}{N} \right)
\]

When \(\ell=1\),
\begin{align*}
C_1^N(s) &= \prod_{j=1}^{N}
\frac{\Gamma\!\left(\tfrac{s}{2} + j\right)^2 }{\Gamma\!\left(\tfrac{s}{2} + 1 + j\right)\Gamma\!\left(\tfrac{s}{2} - 1 + j\right)}\\
&=\prod_{j=1}^N\frac{\frac{s}{2}+j-1}{\frac{s}{2}+j}\\
&=\frac{s}{2N+s}
\end{align*}
and so $\left|C_1^N(s)\right| <1$ since $\Re(s)>-1$ and $N \geq 1$.

Next note that 
\begin{align}
|C_{\ell+1}^N(s)| &= |C_\ell^N(s)| \left| \prod_{j=1}^N \frac{\frac{s}{2}-\ell+j-1}{\frac{s}{2}+\ell+j}\right| \notag \\
&=  |C_1^N(s)| \left| \prod_{j=1}^N \prod_{k=1}^\ell \frac{\frac{s}{2}-k+j-1}{\frac{s}{2}+k+j} \right| \notag\\
& \leq|C_1^N(s)| \left| \prod_{j=1}^N \prod_{k=1}^\ell \frac{\frac{s}{2}-k+j}{\frac{s}{2}+k+j-1} \right| \notag\\
&=|C_1^N(s)||C_{\ell}^N(s)| \notag\\
& \leq|C_1^N(s)|^{\ell+1} .\label{eq:BoundCellInTermsOfC1}
\end{align}

Therefore we can bound the sum of $C_\ell^N(s)$
by summing a geometric series,
\begin{align*}
\sum_{\ell=1}^\infty|C_\ell^N(s)| &\leq \sum_{\ell=1}^\infty |C_1^N(s)|^\ell \\
&= \frac{|s|}{|2N+s|-{|s|}} \\
&= O\left(\frac{1}{N}\right)
\end{align*}
as required.
\end{proof}

\begin{proof}[Proof of Theorem~\ref{thm:second-main}]
We will exclude the case when \(s\) is an even positive number, since as we checked in Theorem~\ref{thm:EvenInteger}, in even cases, the series is finite and every term is strictly positive so the result follows.

When \(s\leq2N,\) we observe that $ C_1^N(s) = s / (2N+s) < \frac12$ so
\[
C_1^N(s) \geq \sum_{\ell=2}^\infty |C_1^N(s)|^\ell
\]
and from \eqref{eq:BoundCellInTermsOfC1} we have
\[
C_1^N(s) \geq \sum_{\ell=2}^\infty |C_\ell^N(s)|
\]
and hence the results follows.

The only remaining case is $s>2N$ and $s$ not an even integer. From the definition of $C_\ell^N(s)$ given in \eqref{eq:CellN}, 
\begin{equation}\label{cellprod}
C_\ell^N(s)=\prod_{j=1}^N \prod_{m=0}
^{\ell-1}\frac{\frac{s}{2}+j-\ell+m}{\frac{s}{2}+j+m}
\end{equation}
which makes it is clear that 
\(C_1^N(s),\cdots,C_{\lfloor s/2\rfloor}^N(s)\) are all positive. Therefore to prove the last remaining case of the theorem, it is sufficient to show that
\[
C_1^N(s) \geq \sum_{\ell=\lfloor s/2\rfloor+1}^\infty |C_\ell^N(s)|.
\]

First, observe that using \eqref{cellprod},
\[
\frac{|C_{\ell+1}|}{|C_\ell|}=\prod_{j=1}^N \frac{|\frac{s}{2}+j-\ell-1|}{\frac{s}{2}+j+\ell}.
\]
When \(s\) is an even integer, the expression for \(C_\ell\) may contain terms of the form \(\frac{0}{0}\), so it should be interpreted by a limiting argument. If \(s\) is not an even integer, no \(C_\ell\) vanishes, and no such issue arises.

Let \(L=\lfloor \frac{s}{2}\rfloor+1 \) and \(\ell=L+n,\;n\geq0.\) We have 
\begin{align*}
    |\frac{s}{2}+j-\ell-1|&=|\frac{s}{2}+j-L-n-1|\\&=|\frac{s}{2}+j-\lfloor \frac{s}{2}\rfloor-1-n-1|\\&<n+j,
\end{align*}

and since \(\frac{s}{2}>N,\;L\geq \frac{s}{2}>N\), we have \(\frac{s}{2}+L+n+j>2N+n+j.\)
Therefore,
\[
\frac{|C_{\ell+1}|}{|C_\ell|}<\prod^N_{j=1}\frac{n+j}{2N+n+j}\leq \frac{n+1}{n+3}.
\]
Thus
\[
|C_{L+n}|<|C_L|\prod^{n-1}_{q=0}\frac{q+1}{q+3}=|C_L|\frac{2}{(n+1)(n+2)}.
\]
Summing gives
\[
\sum_{\ell=L}^\infty |C_\ell|<2|C_L|.
\]

Therefore,
\[
\frac{C_L}{C_{L-1}}\leq  \frac{\frac{s}{2}+1-L}{\frac{s}{2}+L}<\frac{1}{2}
\]
and using 
\[
C_L< \frac{1}{2}C_{L-1}<\frac{1}{2}C_1,
\]
it follows that if $s>0$,
\[
\sum_{\ell=1}^\infty C_\ell^N(s) \geq 0
\]
This means that putting $\phi=0$ into Theorem~\ref{thm:ExactMmtAbsVal} yields
\begin{align*}
\mathbb{E}\left[\frac{1}{N} \sum_m|Z(g_m^0)|^s\right] &= \mathbb{E}[|Z(0)|^s] \left( 1+ 2\sum_{\ell=1}^\infty C_\ell^N(s) \right)\\
&\geq\mathbb{E}[|Z(0)|^s]
\end{align*}
as required.
 \end{proof}

\appendix

\section{Moments without absolute values}

In that case of Gram points, namely when $\phi=0$, the values of the characteristic polynomial are always real. There is some interest in understanding the discrete moments without absolute values. Obviously in the case of even integers, this will be exactly the same as what we calculated before, but for odd integers the size of the moments gives us an indication of how frequently zeta at the Gram points changes sign.

We remark that if Gram's Law held (and we know that it doesn't), then between every two consecutive Gram points there will be exactly one zero of zeta, which means the zeta function will never change sign at the Gram points. This, however, certainly does not happen. The same proof technique as above allows one to calculate the odd moments of the characteristic polynomial at the random matrix Gram points (without taking the absolute value) and they are smaller by a factor of $N^{1/4}$, thus showing the characteristic polynomial must change sign. We will return to these consequences later.

\begin{theorem}
For non-negative integer $s$ and $\phi=0$
\[
\mathbb{E}\left[\frac{1}{N}\sum_m Z(g_m^0)^s \right] = \sum_{\ell=0}^s \prod_{j=1}^N \frac{\Gamma(j+s) \Gamma(j)}{\Gamma(j+s-\ell)\Gamma(j+\ell)}
\]
\end{theorem}
Note that if $s=2k$, since a real number to an even power is positive, this is the same result we proved in Theorem~\ref{thm:EvenInteger}.

If $s=2k+1$, an odd integer, then
\begin{align*}
\mathbb{E}\left[\frac{1}{N}\sum_m Z(g_m^0)^{2k+1} \right] &= \sum_{\ell=0}^{2k+1} \prod_{j=1}^N \frac{\Gamma(j+2k+1) \Gamma(j)}{\Gamma(j+2k+1-\ell)\Gamma(j+\ell)} \\
&=2  \mathbb{E}\left[\left|Z(0)\right|^{2k+1} \right] \sum_{\ell=0}^k \prod_{j=1}^N \frac{\Gamma(j+k+\tfrac12)^2}{\Gamma(j+k+1+\ell)\Gamma(j+k-\ell)}
\end{align*}
where again there is a symmetry between $\ell$ and $2k+1-\ell$, but this time there is no ``central'' term that gets treated specially. 

\begin{remark}
    
For odd positive integers \(s=2k+1\), the comparison of the leading-order contributions reveals a difference of order \(N^{1/4}\). More precisely, in the absence of the absolute value, the main term is of order
\(N^{k(k+1)}\),
whereas the corresponding absolute moment has leading order
\(N^{(k+\frac12)^2}\).

\end{remark}

\section{A heuristic recipe for general integer $k$}

We adapt the heuristic recipe of Conrey--Farmer--Keating--Rubinstein--Snaith \cite{CFKRS05}, to construct (for positive integer $k$) an explicit degree $k^2$ polynomial, $K_k^\phi$, such that
\[
\sum_{g_m(\phi) \leq T } \left| \zeta(\tfrac12+i g_m(\phi)) \right|^{2k} 
=
\int_0^T  K_k^{\phi}\left(\log\frac{t}{2\pi}\right) \frac{1}{2\pi}\log \frac{t}{2\pi} \, dt + O\left(T^{1-\delta}\right)
\]
\begin{remark}
The error term is probably $O(T^{1/2+\epsilon})$ but these things are slightly controversial and not the focus of this paper, so we will stick with keeping this simply as a power saving error.
\end{remark}

\begin{remark}
The main difference with the traditional recipe is in our situation the unmatched chi terms do not vanish; instead they contribute terms like $\cos(2\phi)$, $\cos(4\phi)$ etc.
\end{remark}

Recall that the $g_m(\phi)$ are the phi-Gram points for the Riemann zeta function. For simplicity we will drop the $\phi$ and simply denote these as $g_m$ here. 

Rather than tackling $\left| \zeta(\tfrac12+i g_m )\right|^{2k} $ directly, instead we consider small shifts $\alpha_1,\dots,\alpha_{2k}$. Since Hardy's function, $Z(t)$, is a real function (for $t$ real) satisfying $|Z(t)| = |\zeta(\tfrac12+it)|$, we will instead consider
\[
\sum_{g_m \leq T } \prod_{r=1}^{2k} Z(\alpha_r + g_m) 
\]

\subsection{The modified recipe}

\subsubsection{Step 1: Use the approximate functional equation}

The first step is to replace each of the $2k$ Hardy functions with its approximate functional equation. 

Recall that $\chi(\tfrac12+it) = e^{-2i\theta(t)}$ and
\[
Z(t) = e^{i\theta(t)} \zeta(\tfrac12+it) = e^{-i\theta(t)} \zeta(\tfrac12-it)
\]
and so the approximate functional equation is
\[
Z(\alpha_r+g_m) = e^{i\theta(\alpha_r+g_m)} \sum_{n_r}\frac{1}{n_r^{\frac{1}{2}+i\alpha_r+ig_m}} + e^{-i\theta(\alpha_r+g_m)} \sum_{n_r}\frac{1}{n_r^{\frac{1}{2}-\alpha_r-ig_m}} 
\]
In this heuristic, we ignore all error terms and ignore the length of the sum.

\subsubsection{Step 2: Expand the brackets and keep the diagonal terms}

Expanding $2k$ brackets each containing two terms yields a total of $2^{2k}$ terms.

In the normal recipe (for continuous moments) would keep only the terms where the number of pieces with $e^{+i\theta(t)}$ is exactly equal to the number of pieces $e^{-i\theta(t)}$ , since otherwise the resulting chi factor would oscillate so quickly it would contribute only to the error term, not the main term.

However, for us this is not so! In the case of Gram points  there is no oscillation. This is because $g_m=g_m(\phi)$ are the points where $\theta(g_m)= m\pi+\phi$, and since 
\[
\theta(t) = \frac12 t \log\frac{t}{2\pi e} + O(1) \quad , \quad \theta'(t) = \frac12 \log\frac{t}{2\pi} + O\left(\frac{1}{t^2}\right)
\]
this means for large $m$ and small $\alpha_r$
\[
e^{i \theta(\alpha_r + g_m)} \sim (-1)^m e^{i\phi} \left(\frac{g_m}{2\pi} \right)^{i\alpha_r/2}  .
\]

Therefore, instead we expand and keep all terms, then reduce to the diagonal terms. A typical term will look like
\begin{equation*}
e^{i \phi \sum \epsilon_r} \left(\frac{g_m}{2\pi} \right)^{i\sum \epsilon_r \alpha_r/2} \sum_{n_1} \dots \sum_{n_{2k}}  \left( \prod_{r =1}^{2k} \frac{1}{n_r^{1/2+i\epsilon_r \alpha_r}} \right) 
\times \left( \frac{\prod_{r \in M} n_r}{\prod_{r \in P} n_r} \right)^{i g_m} .
\end{equation*}
Here 
\begin{itemize}
    \item $\epsilon_r=+1$ if we pick from the first piece, $e^{i\theta(\alpha_r+g_m)} \sum_{n_r}\frac{1}{n_r^{\frac{1}{2}+i\alpha_r+ig_m}}$ and $\epsilon_r=-1$ if we pick from the second piece $e^{-i\theta(\alpha_r+g_m)} \sum_{n_r}\frac{1}{n_r^{\frac{1}{2}-i\alpha_r-ig_m}}$ 
    \item $P$ is the set of indices where $\epsilon_r=+1$. 
    \item $M = \{1,\dots,2k\} \setminus P$ is the indices where where $\epsilon_r=-1$.
\end{itemize}

\begin{remark}
    Since there are $2k$ terms, 
\[
\sum_{r=1}^{2k} \epsilon_r = |P|-|M| = 2|P|-2k
\]
is always an even number.
\end{remark}

That is, the typical term looks like
\begin{equation*}
e^{2 i (|P|-k)\phi} \left(\frac{g_m}{2\pi}\right)^{i \sum \epsilon_r \alpha_r /2} \sum_{n_1} \dots \sum_{n_{2k}}  \left( \prod_{r=1}^{2k} \frac{1}{n_r^{1/2+i\epsilon_r \alpha_r}} \right) 
\times \left( \frac{\prod_{r \in M} n_r}{\prod_{r \in P} n_r} \right)^{i g_m} 
\end{equation*}

Each of the terms given above is summed over all $g_m$ up to height $T$. We heuristically assume (as in the traditional recipe) that
\[
\left( \frac{\prod_{r \in M} n_r}{\prod_{r \in P} n_r} \right)^{i g_m} 
\]
oscillates to zero (or at least, an error term) unless 
\[
\prod_{r \in P} n_r = \prod_{r \in M} n_r
\]
which we call the diagonal term.

We now now move to the Euler product, and the diagonal term is
\begin{equation*}
e^{2i(|P|-k)\phi} \left(\frac{g_m}{2\pi}\right)^{i\sum_{r=1}^{2k} \epsilon_r \alpha_r/2 } \prod_{p} \left(\sum_{\substack{u_1 \geq 0, \dots, u_{2k} \geq 0 \\ \sum_{r \in P} u_r = \sum_{r\in M} u_r}} p^{-\sum_{r \in P} u_r - \sum_{r=1}^{2k} i u_r \epsilon_r \alpha_r } \right)
\end{equation*}
where we note that
\[
p^{-(u_1+\dots+u_{2k})/2} = p^{-\sum_{r \in P} u_r / 2 - \sum_{r \in M} u_r/2} = p^{-\sum_{r \in P} u_r}
\]
since we are considering the diagonal term.

\subsubsection{Step 3: Extraction of the polar terms}

Since we are in the diagonal term, it is not possible for there to be just one $r$ with $u_r \neq 0$; there must be at least two distinct indices. The ``worst'' case will be when there is exactly one $r \in P$ and exactly one $s \in M$ with $u_r = 1$ and $u_s=1$, and all the other exponents are zero. This will contribute something looking like $p^{-1 - i \alpha_r + i  \alpha_s}$ (since, in the case when we pick one $r \in P$ and one $s \in M$, we have $\epsilon_r=+1$ and $\epsilon_s=-1$).

That is, the diagonal term looks like
\begin{equation*}
e^{2i(|P|-k)\phi} \left(\frac{g_m}{2\pi}\right)^{i\sum_{r=1}^{2k} \epsilon_r \alpha_r/2 } \prod_{p} \left(1 + \sum_{(r,s) \in (P,M)} p^{-1 - i \alpha_r +i\alpha_s} + \dots \right)
\end{equation*}
where $\dots$ denotes terms like $p^{-2 + \dots}$ and higher powers, all of which are convergent for small $\alpha_r$.

The non-convergent parts can be extracted as zeta functions, since
\begin{multline*}
\mathcal{F}_P(\alpha_1,\dots,\alpha_{2k}) \\
= \prod_{p} \left\{ \left(1 + \sum_{(r,s) \in (P,M)} p^{-1 - i \alpha_r + i \alpha_s} + \dots \right) \prod_{(r,s) \in (P,M)} \left(1 - p^{-1 - i \alpha_r + i \alpha_s} \right) \right\}
\end{multline*}
is convergent, since the term inside the product will look like $1+c p^{-2}$ for some $c \ll p^\epsilon$.

If we define
\[
\mathcal{Z}_P(\alpha_1,\dots,\alpha_{2k}) = \mathcal{F}_P(\alpha_1,\dots,\alpha_{2k}) \prod_{(r,s) \in (P,M)} \zeta(1 + i\alpha_r - i \alpha_s) 
\]
where, as stated above, $P$ is the set of indices where $\epsilon_r=+1$ and $M = \{1,\dots,2k\} \setminus P$, then we have shown that the diagonal term associated to the subset $P$ is
\[
e^{2i(|P|-k)\phi} \left(\frac{g_m}{2\pi}\right)^{i\sum_{r=1}^{2k} \epsilon_r \alpha_r/2} \mathcal{Z}_P(\alpha_1,\dots,\alpha_{2k}) .
\]

In the cases where $P=\emptyset$ or $M=\emptyset$, the product is empty so is given the value $1$.

\begin{remark}
    An alternative expression for the middle piece is
    \[
    \left(\frac{g_m}{2\pi}\right)^{i\sum_{r \in P} \epsilon_r \alpha_r/2 - i \sum_{r \in M} \epsilon_r \alpha_r/2 } 
    \]
    but we won't use this.
\end{remark}

\subsubsection{Step 4: Taking the limit as the shifts tend to zero.}

To summarise, we have shown
\begin{equation*}
    \sum_{g_m \leq T } \prod_{r=1}^{2k} Z(\alpha_r+ g_m)
    =  \sum_{g_m \leq T } \left\{ \sum_{\ell=-k}^k e^{2 i \ell \phi} \sum_{\substack{P \subseteq \{1,\dots,2k\} \\ |P|=k+\ell}} \left(\frac{g_m}{2\pi}\right)^{i\sum \epsilon_r \alpha_r/2 } \mathcal{Z}_P(\alpha_1,\dots,\alpha_{2k}) \right\}
\end{equation*}

Letting $(\alpha_1,\dots,\alpha_{2k}) \to (0,\dots,0)$ introduces a large number of poles. However, they all cancel each other out.

For each $ \ell \in \{-k,\dots, k\}$, we have
\[
\sum_{\substack{P \subseteq \{1,\dots,2k\} \\ |P|=k+\ell}} \left(\frac{g_m}{2\pi}\right)^{i\sum \epsilon_r \alpha_r /2} \mathcal{Z}_P(\alpha_1,\dots,\alpha_{2k}) \to \widetilde L^\ell_k\left(\log \frac{g_m}{2\pi} \right)
\]
where $L_k^\ell$ is a polynomial of degree $k^2-\ell^2$ (since $|P| = k+\ell$, so $|M|=k-\ell$), and we observe both that $L_k^0 = \widetilde K_k$, the kernel from the continuous $2k$th moment, and also that $L_k^\ell = L_k^{-\ell}$ (they are complex conjugates of each other which are real when the limit is taken).

Defining
\[
K_k^\phi(x) = \widetilde K_k(x) + \sum_{\ell=1}^k 2 \cos(2\ell\phi) L_k^\ell(x)
\]
we have shown
\begin{align*}
\sum_{g_m \leq T} \left|\zeta(\tfrac12+i g_m)\right|^{2k} &= \sum_{g_m \leq T} Z(g_m)^{2k} \\
&= \sum_{g_m \leq T} K_k^\phi\left(\log \frac{g_m}{2\pi}\right) \\
&=\int_0^T \frac{1}{2\pi} \log \frac{t}{2\pi} K_k^\phi\left(\log \frac{t}{2\pi}\right) dt
\end{align*}
since the sum on the RHS depends smoothly on the Gram points, and they are very regularly spaced, occurring with density $\frac{1}{2\pi} \log \frac{t}{2\pi}$, enabling us to move from a sum to an integral.

\subsection{Example: $k=1$}

We have
\begin{center}
    \begin{tabular}{|c|c|c|}
    \hline
    $\ell$ & $P$ & $\mathcal{Z}_P(\alpha_1,\alpha_{2})$  \\
    \hline
    $-1$ & $\emptyset$ &  $1$ \\
    $0$ & $\{1\}$ & $\zeta(1+i\alpha_1-i\alpha_2)$ \\
    $0$ & $\{2\}$ & $\zeta(1+i\alpha_2-i\alpha_1)$ \\
    $1$ & $\{1,2\}$ & $1$ \\
    \hline
     \end{tabular}
\end{center}

Therefore
\[
L_1^1(x) =  \lim_{\substack{\alpha_1 \to 0 \\ \alpha_2 \to 0}} e^{ix(\alpha_1+\alpha_2)/2} = 1 =  \lim_{\substack{\alpha_1 \to 0 \\ \alpha_2 \to 0}} e^{ix(-\alpha_1-\alpha_2)/2}  = L_1^{-1}(x) 
\]
\[
L_1^1(x) =  \lim_{(\alpha_1,\alpha_2) \to (0,0)} e^{ix(\alpha_1+\alpha_2)/2} = 1 =  \lim_{(\alpha_1,\alpha_2) \to (0,0)} e^{ix(-\alpha_1-\alpha_2)/2}  = L_1^{-1}(x) 
\]
and
\begin{align*}
L_1^0(x) &= \lim_{\substack{\alpha_1 \to 0 \\ \alpha_2 \to 0}} e^{ix(\alpha_1-\alpha_2)/2} \zeta(1+i\alpha_1-i\alpha_2) + e^{ix(-\alpha_1+\alpha_2)/2} \zeta(1+i\alpha_2-i\alpha_1) \\
&=x+2 \gamma
\end{align*}

In conclusion,
\[
K_1^\phi(x) = x + 2\gamma + 2\cos(2\phi)
\]
in perfect agreement with \eqref{eq:Kernelk1} which is the translation of the result of Kalpokas and Stueding~\cite{KalSte09} into an integrable kernel.

However, our result yields more than that, stating
\begin{multline*}
\sum_{g_m \leq T} Z(g_m+\alpha_1)Z(g_m+\alpha_2) \\
= \int_0^T \frac{1}{2\pi} \log \frac{t}{2\pi} \left(2 \Re \left\{\left(\frac{t}{2\pi}\right)^{i(\alpha_1-\alpha_2)/2} \zeta(1+i\alpha_1-i\alpha_2) \right\} + 2\cos(2\phi+\tfrac12\alpha_1+\tfrac12\alpha_2) \right) dt \\
+ O\left(T^{1-\delta}\right)
\end{multline*}

\subsection{Example: $k=2$}\label{sect:k2Appendix}

\begin{center}
    \begin{tabular}{|c|c|c|}
    \hline
    $\ell$ & $P$ & $\mathcal{Z}_P(\alpha_1,\alpha_2,\alpha_3,\alpha_4)$  \\
    \hline
    $-2$ & $\emptyset$ &  $1$ \\
    $-1$ & $\{1\}$ & $\zeta(1+i\alpha_1-i\alpha_2) \zeta(1+i\alpha_1-i\alpha_3) \zeta(1+i\alpha_1-i\alpha_4)$ \\
    $-1$ & $\{2\}$ &  $\zeta(1+i\alpha_2-i\alpha_1) \zeta(1+i\alpha_2-i\alpha_3) \zeta(1+i\alpha_2-i\alpha_4)$ \\
    $-1$ & $\{3\}$ &  $\zeta(1+i\alpha_3-i\alpha_1) \zeta(1+i\alpha_3-i\alpha_2) \zeta(1+i\alpha_3-i\alpha_4)$\\
    $-1$ & $\{4\}$ & $\zeta(1+i\alpha_4-i\alpha_1) \zeta(1+i\alpha_4-i\alpha_2) \zeta(1+i\alpha_4-i\alpha_3)$ \\
    $0$ & $\{1,2\}$ & $\frac{\zeta(1+i\alpha_1-i\alpha_3) \zeta(1+i\alpha_1-i\alpha_4) \zeta(1+i\alpha_2-i\alpha_3) \zeta(1+i\alpha_2-i\alpha_4)}{\zeta(2+i\alpha_1+i\alpha_2-i\alpha_3-i\alpha_4)}$ \\
    $0$ & $\{1,3\}$ & $\frac{\zeta(1+i\alpha_1-i\alpha_2) \zeta(1+i\alpha_1-i\alpha_4) \zeta(1+i\alpha_3-i\alpha_2) \zeta(1+i\alpha_3-i\alpha_4)}{\zeta(2+i\alpha_1-i\alpha_2+i\alpha_3-i\alpha_4)}$\\
    $0$ & $\{1,4\}$ & $\frac{\zeta(1+i\alpha_1-i\alpha_2) \zeta(1+i\alpha_1-i\alpha_3) \zeta(1+i\alpha_4-i\alpha_2) \zeta(1+i\alpha_4-i\alpha_3)}{\zeta(2+i\alpha_1-i\alpha_2-i\alpha_3+i\alpha_4)}$\\
    $0$ & $\{2,3\}$ & $\frac{\zeta(1+i\alpha_2-i\alpha_1) \zeta(1+i\alpha_2-i\alpha_4) \zeta(1+i\alpha_3-i\alpha_1) \zeta(1+i\alpha_3-i\alpha_4)}{\zeta(2-i\alpha_1+i\alpha_2+i\alpha_3-i\alpha_4)}$\\
    $0$ & $\{2,4\}$ & $\frac{\zeta(1+i\alpha_2-i\alpha_1) \zeta(1+i\alpha_2-i\alpha_3) \zeta(1+i\alpha_4-i\alpha_1) \zeta(1+i\alpha_4-i\alpha_3)}{\zeta(2-i\alpha_1+i\alpha_2-i\alpha_3+i\alpha_4)}$\\
    $0$ & $\{3,4\}$ & $\frac{\zeta(1+i\alpha_3-i\alpha_1) \zeta(1+i\alpha_3-i\alpha_2) \zeta(1+i\alpha_4-i\alpha_1) \zeta(1+i\alpha_4-i\alpha_2)}{\zeta(2-i\alpha_1-i\alpha_2+i\alpha_3+i\alpha_4)}$ \\
    $1$ & $\{2,3,4\}$ & $\zeta(1+i\alpha_2-i\alpha_1) \zeta(1+i\alpha_3-i\alpha_1) \zeta(1+i\alpha_4-i\alpha_1)$ \\
    $1$ & $\{1,3,4\}$ & $\zeta(1+i\alpha_1-i\alpha_2) \zeta(1+i\alpha_3-i\alpha_2) \zeta(1+i\alpha_4-i\alpha_2)$  \\
    $1$ & $\{1,2,4\}$ & $\zeta(1+i\alpha_1-i\alpha_3) \zeta(1+i\alpha_2-i\alpha_3) \zeta(1+i\alpha_4-i\alpha_3)$  \\
    $1$ & $\{1,2,3\}$ & $\zeta(1+i\alpha_1-i\alpha_4) \zeta(1+i\alpha_2-i\alpha_4) \zeta(1+i\alpha_3-i\alpha_4)$ \\
    $2$ & $\{1,2,3,4\}$ & $1$ \\
    \hline
     \end{tabular}
\end{center}

Therefore we have
\[
L_2^{-2}(x) = L_2^2(x)  =1
\]
and
\begin{align*}
L_2^{-1}(x) &= \lim_{\substack{\alpha_1 \to 0\\\alpha_2 \to 0\\ \alpha_3 \to 0\\ \alpha_4 \to 0}} \biggl(e^{ix (\alpha_1-\alpha_2-\alpha_3-\alpha_4)/2} \mathcal{Z}_{\{1\}} + e^{ix (-\alpha_1+\alpha_2-\alpha_3-\alpha_4)/2} \mathcal{Z}_{\{2\}}\\
&\quad\quad\quad\quad   + e^{ix (-\alpha_1-\alpha_2+\alpha_3-\alpha_4)/2} \mathcal{Z}_{\{3\}}  +  e^{ix (-\alpha_1-\alpha_2-\alpha_3+\alpha_4)/2} \mathcal{Z}_{\{4\}} \biggr) \\
&=\frac{1}{6} x^3 + 2 \gamma x^2 + \left(6 \gamma ^2-4 \gamma _1\right) x -12 \gamma 
   \gamma _1+2 \gamma _2+4 \gamma ^3
\end{align*}
and
\begin{align*}
L_2^{1}(x) &= \lim_{\substack{\alpha_1 \to 0\\\alpha_2 \to 0\\ \alpha_3 \to 0\\ \alpha_4 \to 0}} \biggl(e^{ix (-\alpha_1+\alpha_2+\alpha_3+\alpha_4)/2} \mathcal{Z}_{\{2,3,4\}} + e^{ix (\alpha_1-\alpha_2+\alpha_3+\alpha_4)/2} \mathcal{Z}_{\{1,3,4\}} \\
&\quad\quad\quad\quad + e^{ix (\alpha_1+\alpha_2-\alpha_3+\alpha_4)/2} \mathcal{Z}_{\{1,2,4\}} +   e^{ix (\alpha_1+\alpha_2+\alpha_3-\alpha_4)/2} \mathcal{Z}_{\{1,2,3\}} \biggr) \\
&=\frac{1}{6} x^3 + 2 \gamma x^2 + \left(6 \gamma ^2-4 \gamma _1\right) x -12 \gamma 
   \gamma _1+2 \gamma _2+4 \gamma ^3
\end{align*}
That the two have the same limit is not surprising, since if we replace the $\mathcal{Z}$s with the zeta functions, then we see the two expressions are the same with $\alpha_1 \leftrightarrow \alpha_3$ and  $\alpha_2 \leftrightarrow \alpha_4$. This limit is the same as that given in \eqref{eq:L2}.

Finally, we recover the usual kernel for the continuous fourth moment
\begin{align*}
L_2^{0}(x) &= \lim_{\substack{\alpha_1 \to 0\\\alpha_2 \to 0\\ \alpha_3 \to 0\\ \alpha_4 \to 0}} \biggl(e^{ix (\alpha_1+\alpha_2-\alpha_3-\alpha_4)/2} \mathcal{Z}_{\{1,2\}} + e^{ix (\alpha_1-\alpha_2+\alpha_3-\alpha_4)/2}  \mathcal{Z}_{\{1,3\}} \\
&\quad\quad\quad\quad  + e^{ix (\alpha_1-\alpha_2-\alpha_3+\alpha_4)/2}  \mathcal{Z}_{\{1,4\}} + e^{ix (-\alpha_1+\alpha_2+\alpha_3-\alpha_4)/2}  \mathcal{Z}_{\{2,3\}} \\
&\quad\quad\quad\quad   + e^{ix (-\alpha_1+\alpha_2-\alpha_3+\alpha_4)/2}  \mathcal{Z}_{\{2,4\}}  + e^{ix (-\alpha_1-\alpha_2+\alpha_3+\alpha_4)/2} \mathcal{Z}_{\{3,4\}} \biggr) \\
&= \frac{1}{2 \pi ^2} x^4 + \left(\frac{8 \gamma }{\pi ^2}-\frac{24 \zeta '(2)}{\pi ^4} \right) x^3 + a_2 x^2 + a_1 x + a_0
\end{align*}
with
\[
a_2 = -\frac{72 \zeta ''(2)}{\pi ^4}-\frac{288 \gamma  \zeta '(2)}{\pi
   ^4}+\frac{864 \zeta '(2)^2}{\pi ^6}-\frac{12 \gamma _1}{\pi
   ^2}+\frac{42 \gamma ^2}{\pi ^2}
\]
\begin{multline*}
a_1 = -\frac{96 \zeta ^{(3)}(2)}{\pi ^4}-\frac{576 \gamma  \zeta
   ''(2)}{\pi ^4}-\frac{1008 \gamma ^2 \zeta '(2)}{\pi
   ^4}+\frac{6912 \gamma  \zeta '(2)^2}{\pi ^6}-\frac{20736 \zeta
   '(2)^3}{\pi ^8}+\frac{3456 \zeta '(2) \zeta ''(2)}{\pi
   ^6}\\
   +\frac{288 \gamma _1 \zeta '(2)}{\pi ^4}-\frac{120 \gamma 
   \gamma _1}{\pi ^2}-\frac{12 \gamma _2}{\pi ^2}+\frac{72 \gamma
   ^3}{\pi ^2}
\end{multline*}
and
\begin{multline*}
a_0 = -\frac{48 \zeta ^{(4)}(2)}{\pi ^4}-\frac{384 \gamma  \zeta
   ^{(3)}(2)}{\pi ^4}-\frac{1008 \gamma ^2 \zeta ''(2)}{\pi
   ^4}+\frac{1728 \zeta ''(2)^2}{\pi ^6}-\frac{864 \gamma ^3
   \zeta '(2)}{\pi ^4}+\frac{12096 \gamma ^2 \zeta '(2)^2}{\pi
   ^6} \\
   -\frac{82944 \gamma  \zeta '(2)^3}{\pi ^8}+\frac{248832
   \zeta '(2)^4}{\pi ^{10}}+\frac{2304 \zeta ^{(3)}(2) \zeta
   '(2)}{\pi ^6}+\frac{13824 \gamma  \zeta '(2) \zeta ''(2)}{\pi
   ^6} \\ 
   -\frac{62208 \zeta '(2)^2 \zeta ''(2)}{\pi ^8}
   +\frac{288
   \gamma _1 \zeta ''(2)}{\pi ^4}+\frac{1440 \gamma  \gamma _1
   \zeta '(2)}{\pi ^4}+\frac{144 \gamma _2 \zeta '(2)}{\pi
   ^4}-\frac{3456 \gamma _1 \zeta '(2)^2}{\pi ^6} \\
   -\frac{144
   \gamma ^2 \gamma _1}{\pi ^2} 
   +\frac{84 \left(\gamma
   _1\right){}^2}{\pi ^2}+\frac{12 \gamma  \gamma _2}{\pi
   ^2}+\frac{20 \gamma _3}{\pi ^2}+\frac{36 \gamma ^4}{\pi ^2}
\end{multline*}

Putting all this together,
\[
\sum_{g_m \leq T} \left|\zeta(\tfrac12+i g_m)\right|^{4}  =\int_0^T \frac{1}{2\pi} \log \frac{t}{2\pi} K_2^\phi\left(\log \frac{t}{2\pi}\right) dt + O\left(T^{1-\epsilon}\right)
\]
with
\[
K_2^\phi(x) = L_2^{0}(x) + 2\cos(2\phi) L_2^{1}(x) + 2\cos(4\phi)
\]

\section*{Acknowledgments and AI declaration}

This research will form part of the second author's PhD thesis. The work was carried out during 2025/26, and AI was not used in the writing of this paper, apart from a literature search.

\bibliography{bibliog}    
\bibliographystyle{plain}

\end{document}